\documentclass{amsart}

\usepackage{latexsym,amsmath,amssymb,amsfonts,amscd,graphics,appendix,amsxtra}
\usepackage[mathscr]{eucal}
\usepackage[all]{xypic}
\usepackage[normalem]{ulem}

\newtheorem{mainthm}{Theorem}
\numberwithin{equation}{section}
\theoremstyle{plain}
\newtheorem{theorem}[equation]{Theorem}

\newtheorem{lemma}[equation]{Lemma}

\newtheorem{proposition}[equation]{Proposition}
\newtheorem{corollary}[equation]{Corollary}

\theoremstyle{remark}
\newtheorem{remark}[equation]{Remark}
\newtheorem*{remark*}{Remark}

\theoremstyle{definition}
\newtheorem{definition}[equation]{Definition}

\newcommand{\D}{\displaystyle}

\def\Zee{\mathbb{Z}}
\def\Q{\mathbb{Q}}
\def\R{\mathbb{R}}
\def\Cee{\mathbb{C}}

\def\Id{\operatorname{Id}}

\def\End{\operatorname{End}}

\newcommand{\bQ}{\mathbb{Q}}
\newcommand{\bR}{\mathbb{R}}
\newcommand{\bC}{\mathbb{C}}

\newcommand{\Res}{\mathrm{Res}}

\newcommand{\Nm}{\mathrm{Nm}}

\newcommand{\Spin}{\mathrm{Spin}}

\newcommand{\SO}{\mathrm{SO}}
\newcommand{\Sp}{\mathrm{Sp}}
\newcommand{\Gal}{\mathrm{Gal}}

\newcommand{\SU}{\mathrm{SU}}

\newcommand{\GL}{\mathrm{GL}}

\newcommand{\Hg}{\mathrm{Hg}}

\newcommand{\calD}{\mathcal{D}}

\newcommand{\sX}{\mathscr{X}}

\newcommand{\calV}{\mathscr{V}}

\newcommand{\I}{\mathrm{I}}
\newcommand{\II}{\mathrm{II}}

\title[Hermitian variations of Calabi--Yau type with real multiplication]{On some  Hermitian variations of Hodge structure of Calabi--Yau type with real multiplication} 
\date{\today}
\begin{document}

\author[R. Friedman]{Robert Friedman}
\address{Columbia University, Department of Mathematics, New York, NY 10027}
\email{rf@math.columbia.edu}
\author[R. Laza]{Radu Laza}
\address{Stony Brook University, Department of Mathematics,  Stony Brook, NY 11794}
\email{rlaza@math.sunysb.edu}
\thanks{The  second author was partially supported by NSF grant DMS-1200875 and a Sloan Fellowship}

\begin{abstract}
We prove that, for every totally real number field $E_0$, there exists a $\bQ$-variation of Hodge structure $\calV$ of Calabi--Yau threefold type  with associated endomorphism algebra $E_0$   such that the unique irreducible factor of  Calabi-Yau type of $\calV_\bR$ is the canonical $\bR$-VHS of CY type over the Hermitian symmetric domain $\II_6$, associated to the real group $\SO^*(12)$. The main point is a rationality result for the half spin representations of a form of the group $\SO^*(4m)$ defined over a number field.
\end{abstract}
\bibliographystyle{alpha}
\maketitle

\section*{Introduction}
A Hodge structure of \textsl{Calabi--Yau or CY type} is an effective weight $n$ Hodge structure with $h^{n,0}=1$.   Work of Gross \cite{bgross} and Sheng--Zuo \cite{sz} shows  that every Hermitian symmetric domain $\calD$ carries a canonical $\bR$-variation of Hodge structure $\calV$ of CY type (cf.\ also \cite[\S2]{FL} for more discussion). Furthermore, every other equivariant $\bR$-VHS (or \textsl{Hermitian VHS}) of CY type on $\calD$ is obtained from $\calV$ using certain standard constructions (see \cite[Theorem 2.22]{FL}). For example, each of the four rank $3$ Hermitian symmetric tube domains $\calD$, namely $\mathrm{III}_3$, $\mathrm{I}_{3,3}$, $\mathrm {II}_6$, and $\mathrm{EVII}$ (corresponding to the real Lie groups $\Sp(6,\bR)$, $\SU(3,3)$, $\SO^*(12)$, and $\mathrm{E}_{7,3}$ respectively), carries a weight $3$  $\bR$-VHS of CY type,  with the relevant Hodge number $h^{2,1}=6$, $9$, $15$, and $27$ respectively, and every primitive irreducible weight $3$ Hermitian VHS of CY type which is also of tube type is of this form.  Here primitive means that the VHS is not induced from a lower weight VHS in an obvious sense, and tube type means that the corresponding complex VHS is irreducible. This gives a satisfactory classification (over $\bR$) of   Hermitian VHS of CY type analogous to the classification of Satake \cite{satake} and Deligne \cite{dshimura} of totally geodesic holomorphic embeddings of Hermitian symmetric domains into the Siegel upper half space $\frak H_g$, or equivalently Hermitian VHS of abelian variety type. 

The analogous classification over $\bQ$ of  Hermitian VHS $\calV$ of Calabi-Yau type is much more difficult. The weight $2$ case, or \textsl{$K3$ type}, was analyzed by Zarhin \cite{zarhin} and van Geemen \cite{vgk3}. A basic invariant measuring the difference between the classification over $\bQ$ and over $\bR$ is the algebra  $E:=\End_{\Hg}(V_s)$ of Hodge endomorphisms  of a general fiber $V_s$ of $\calV$. In the Calabi-Yau type case, $E$ is either a totally real field or a CM field (see \cite{zarhin} or \cite[Prop.\  3.1]{FL}). If $E=E_0$ is a totally real field, we say that the Hermitian VHS has \textsl{weak real multiplication by $E_0$}. In weight $3$, we showed in \cite[Theorem 3.18]{FL} that there are at most two primitive cases of Hermitian VHS of CY threefold type defined over $\bQ$ with non-trivial weak real multiplication.  These  cases correspond to the domains $\I_{3,3}$ and $\II_6$ associated to the groups  $\SU(3,3)$  and $\SO^*(12)$ respectively. In the two other tube domain cases mentioned above, $\mathrm{III}_3$ and $\mathrm{EVII}$,   non-trivial weak real multiplication cannot arise. 
For the $\SU(3,3)$ case, we showed in \cite{FL} that every totally real field $E_0$ can be realized as the endomorphism algebra of a Hermitian VHS of CY type defined over $\bQ$. This result is similar to that of van Geemen \cite{vgk3} for $K3$ type,  but the representation theory is more involved. This paper is devoted to the remaining case of the group $\SO^*(12)$ and  corresponding Hermitian symmetric space $\II_6$. More precisely, we prove the following (using freely the notation of \cite{FL}):

\begin{mainthm} Let $E_0$ be a totally real field with $d = [E_0:\Q]$. Then there exists an almost simple $\Q$-group $G$ and an irreducible weight three Hodge representation $\rho\colon G \to \GL(V)$ with a corresponding weight three Hermitian VHS $\calV$ such that the generic endomorphism algebra for $\calV$ has real multiplication by $E_0$ and, over $\R$, $\calV = \bigoplus_{i=1}^d\calV_i$, where $\calV_1$  corresponds to a half spin representation of the spin double cover of the real group $\SO^*(12)$ and, for   $i>1$,  $\calV_i$    corresponds  to (a Tate twist of) a  half spin representation of the spin double cover of the real group $\SO(2, 10)$.
\end{mainthm} 

\begin{remark*} With notations as in the theorem, $\calV_1$ is a VHS of CY threefold type with Hodge numbers $(1,15,15,1)$. For $i>1$,  $\calV_i$ is (up to a Tate twist) a VHS of abelian variety type, with Hodge numbers $(0,16,16,0)$.  
\end{remark*}
Arguing as in the proof of Theorem 3.18 in \cite{FL}, it will be enough to show:

\begin{mainthm}\label{mainthm2} Let $E_0$ be a totally real field with $d = [E_0:\Q]$, and let $\sigma_1, \dots, \sigma_d$ be the different embeddings of $E_0$ in $\R$. We view $E_0$ as a subfield of $\R$ via $\sigma_1$. Then there exists a geometrically almost simple group $G_1$ defined over $E_0$ and a representation $\rho_1\colon G_1 \to \GL(V_1)$, where $V_1$ is an $E_0$-vector space, such that
\begin{enumerate}
\item[\rm(i)] The induced complex group $G_{1, \Cee} \cong \Spin(12)$ and $V_{1, \Cee}$ is the half spin representation $S^+$ of $\SO(12)$, in the notation of \cite{fultonharris}. 
\item[\rm(ii)] The induced real group $G_{1, \R}$ is isomorphic to the spin  double cover of $\SO^*(12)$.
\item[\rm(iii)] For $i> 1$, viewing $E_0$ as a subfield of $\R$ via the embedding $\sigma_i$, the induced real group $G_{i, \R}$ is isomorphic to the spin  double cover of $\SO(2, 10)$.
\end{enumerate} 
\end{mainthm}

As in \cite{FL}, once we have the group $G_1$ and the representation $V_1$, we take $G=\Res_{E_0/\bQ} G_1$ together with the induced representation on $V=\Res_{E_0/\bQ} V_1$. Then $G_\R = \prod_{i=1}^dG_{i, \R}$,  and   $V\otimes_{\Q}\R = \bigoplus _{i=1}^d V_{i, \R}$ where $V_{i,\R}$ is the real vector space $V_1\otimes _{E_0, \sigma_i}\R$.
The associated $\bQ$-VHS $\calV$ will have real multiplication by $E_0$ and will split over $\bR$  into a CY piece $\calV_1$ and a product of Tate twists of weight $1$ factors coming from the Kuga-Satake construction applied to $\SO(2, 10)$, say  $\calV_2, \dots, \calV_d$. 
The Mumford-Tate domain is $\calD_1\times \calD_2\times \cdots \times \calD_d$, a product of  different Hermitian symmetric domains, with $\calD_i$ parametrizing the Hermitian VHS $\calV_i$, as desired.

Most of the VHS of CY type occurring in geometry are not of Hermitian type. Still, the Hermitian VHS are interesting, as they provide simple test cases for mirror symmetry. For example, we mention the examples of VHS of CY threefolds without maximal monodromy due to Rohde \cite{rohde3}, which  are both geometric and of Hermitian type. Thus, an interesting question is whether the (abstract) VHS of CY type discussed here can be realized geometrically or motivically. The case $\mathrm{III}_3$ is   the case of abelian threefolds. Recently, Zheng Zhang has shown that the case  $\mathrm{I}_{3,3}$, including the case of non-trivial real multiplication, can be realized motivically starting from a VHS of abelian variety type. At the other extreme, Deligne showed in \cite{dshimura} that the case $\mathrm{EVII}$ cannot be realized motivically starting from a VHS of abelian variety type, and it is a well known open problem to give some motivic realization for it. It is possible that the  $\mathrm {II}_6$ case has a motivic realization starting with a VHS of abelian varieties, but Zheng Zhang has also noted that such a construction cannot work in the presence of non-trivial real multiplication. 

The outline of this paper is as follows. In Section 1, we discuss some general forms of the group $\SO^*(2n)$ defined over totally real  fields. As a warm up to the proof of Theorem~\ref{mainthm2}, we give a necessary and sufficient condition for  the rationality of the standard representation. The proof of Theorem~\ref{mainthm2} then consists of two steps. In the first step, in Section 2, given an imaginary quadratic extension $E$ of $E_0$, we construct the group $G_1$ and a representation of $G_1$ defined over $E$ whose tensor product with $\Cee$   is a half spin representation. Section 3 describes the second and more difficult step:  We prove that the representation of the group $G_1$ is actually defined over $E_0$, as part of a more general discussion of the rationality of the half spin representations for the groups described in Sections 1 and 2. We note that, over $\bR$, the fact that the relevant half spin representation is defined over $\bR$ (as opposed to only over $\bC$)  follows  from the general criterion of \cite[Theorem IV.E.4]{mtbook}.

\section{Preliminaries}

Throughout this paper, $E_0$ will either be a totally real number field or $\R$, and $E$ will   be an imaginary quadratic extension of $E_0$, hence $E=\Cee$ in case $E_0 =\R$. If $\sigma\in \Gal(E/E_0)$ is the nontrivial element, we denote $\sigma(\alpha)$ by $\bar{\alpha}$. If $V_1$, $V_2$ are two $E$-vector spaces, an additive homomorphism $f\colon V_1\to V_2$ is \textsl{conjugate linear} if, for all $v\in V_1$ and $\alpha \in E$, $f(\alpha v) = \bar{\alpha}f(v)$. In particular, $f$ is $E_0$-linear. Our goal in this section will be to describe the linear algebra necessary to construct a form of the real group $\SO^*(2n)$ over $E_0$ (as defined in \cite{Knapp}), and also to discuss the rationality of the standard representation. Let $W$ be an $E$-vector space of dimension $2n$ with $E$-basis $e_1, \dots, e_{2n}$. We write $z = \sum_{i=1}^{2n}z_ie_i$, and similarly for $w$. Suppose that $b(z,w)$ is a nondegenerate $E$-bilinear form on $W$, written in the standard form 
$$b(z,w) = \sum_{i=1}^n (z_iw_{n+i}+ z_{n+i}w_i),$$
in other words $b(e_i, e_j) = 0$ if $i,j\leq n$ or $i,j\geq n+1$, and $b(e_i, e_{n+j}) = \delta_{ij}$. Let $\psi$ be  an $(E, E_0)$-Hermitian form on $W$. For the moment, we will just make the assumption that $\psi$ is diagonalized in the basis $e_1, \dots, e_{2n}$ and write
$$\psi(z,w) = \sum_{i=1}^{2n}a_iz_i\bar{w}_i,$$
where necessarily the $a_i \in E_0$.
Then $b$ defines an $E$-linear isomorphism $W\to W\spcheck$,   denoted by $B$, via the rule $B(v)(w) = b(v,w)$, and $\psi$ defines  a conjugate linear isomorphism $W\to W\spcheck$,   denoted by $\Psi$, via the rule $\Psi(v)(w) = \psi(w,v)$. Finally define
$$J =B^{-1}\circ \Psi.$$

\begin{lemma} In the above notation, $J$ is specified by the properties that $J$ is conjugate linear, and, for all $i\leq n$,
\begin{align*}
J(e_i) & = a_ie_{n+i};\\
J(e_{n+i}) & = a_{n+i}e_i.
\end{align*}
\end{lemma}
\begin{proof} If $e_i^*$ are the dual basis vectors in $W\spcheck$ (i.e.\ $e_i^*(e_j) = \delta_{ij}$), then clearly, for $i\leq n$, $B(e_i) = e_{n+i}^*$ and $B(e_{n+i}) = e_i^*$. Moreover $\Psi(e_i) = a_ie_i^*$ for all $i$. The proof is then immediate.
\end{proof}

\begin{corollary} In the above notation, $J^2$ is specified by the properties that $J$ is $E$-linear, and that, for all $i\leq n$,
\begin{align*}
J^2(e_i) & = a_ia_{n+i}e_i;\\
J^2(e_{n+i}) & = a_ia_{n+i}e_{n+i}. \qed
\end{align*}
\end{corollary}

In particular, we see that $J^2 = \lambda \Id$ for some $\lambda \in E_0$ $\iff$ $a_ia_{n+i} = \lambda$ for all $i\leq n$, i.e.\ there exists a $\lambda\in E_0$ such that, for $i\leq n$, $a_{n+i} = \lambda a_i^{-1}$. In terms of the forms, this says:

\begin{lemma} There exists a $\lambda \in E_0$ such  that $J^2 = \lambda \Id$ $\iff$ $\lambda (\Psi^{-1}\circ B) = B^{-1}\circ \Psi$. 
\end{lemma}
\begin{proof} $J^2 = \lambda \Id$ $\iff$ $J = \lambda J^{-1}$ $\iff$ $B^{-1}\circ \Psi = \lambda (\Psi^{-1}\circ B)$. 
\end{proof}

Let $\Phi\colon W \to W'$ be an $E$-linear isomorphism. Then $\Phi$ defines a quadratic form $b_\Phi$ on $W'$ via: $b_\Phi(\xi, \eta) = b(\Phi^{-1}(\xi), \Phi^{-1}(\eta))$. Equivalently, $b(v,w) = b_{\Phi}(\Phi(v), \Phi(w))$, so that $\Phi\colon (W,b) \to (W', b_{\Phi})$ is an isomorphism of quadratic spaces. For example, $B$ defines a form on $W\spcheck$, the dual quadratic form, which we just denote by $b\spcheck$. A  conjugate linear isomorphism $\Psi \colon W \to W'$ defines an $E$-bilinear form as well, which we denote by $\bar{b}_\Psi$, via the rule
$$\bar{b}_\Psi(\xi, \eta) = \overline{b(\Psi^{-1}(\xi), \Psi^{-1}(\eta))}.$$

Direct calculation then shows:

\begin{lemma}\label{lemlambda} With $b, \psi, B, \Psi$ as above, there exists a $\lambda\in E_0$ such that 
$\bar{b}_\Psi = \lambda^{-1} b\spcheck$ $\iff$ $a_ia_{n+i} = \lambda$ for all $i\leq n$. \qed
\end{lemma}

\begin{definition}\label{goodcomp} 1) A pair of forms $b$, $\psi$ satisfying either of the equivalent conditions of the previous lemma will be called \textsl{compatible}.

\smallskip

\noindent 2) With $b, \psi$ compatible as above, let $e_1, \dots, e_{2n}$ be an $E$-basis of $W$ such that $b(z,w) = \sum_{i=1}^n (z_iw_{n+i}+ z_{n+i}w_i)$ and $\psi(z,w) = \sum_{i=1}^{2n}a_iz_i\bar{w}_i$, where as before we write $z=\sum_iz_ie_i$ and $w=\sum_iw_ie_i$. We call $e_1, \dots, e_{2n}$ a \textsl{good basis} if there exists a $\lambda \in E_0$ such that $a_ia_{n+i} = \lambda$ for all $i$, $1\leq i \leq n$. If $e_1, \dots, e_{2n}$ is a good basis, we call the maximal $b$-isotropic subspace $W_1$ which is the span over $E$ of $e_1, \dots, e_n$ a \textsl{good isotropic subspace} of $W$. 

\smallskip

\noindent 3) For a compatible $b, \psi$, a good basis $e_1, \dots, e_{2n}$, and a good isotropic subspace $W_1$ of $W$, we denote by $D= \det(\psi|W_1)$ the discriminant $a_1\cdots a_n$ of the Hermitian form $\psi|W_1$ with respect to the basis $e_1, \dots, e_n$ of $W_1$. (More invariantly, $\det(\psi|W_1)$ is only well-defined up to a norm, i.e.\  as an element of $E_0^*/\Nm_{E/E_0}(E^*)$.)
\end{definition}

For an $E$-vector space $W$, we write $\Res_{E/E_0}W$ for the Weil restriction of scalars of $W$: $\Res_{E/E_0}W$ is just $W$ considered as an $E_0$-vector space. For an algebraic group $G$ defined over $E$, the restriction of scalars $\Res_{E/E_0}G$ is similarly defined (see also \cite[\S I.4.b]{milneala}), and is an algebraic group over $E_0$.

Let $G_0 = G(W, b, \psi)$ be the set of $E$-linear isomorphisms of $W$ preserving $b$ and $\psi$. The group $G_0$ is the group of $E_0$-valued points of an affine algebraic group over $E_0$, also denoted by $G_0$: It is the intersection of the algebraic group $\Res_{E/E_0}\SO(W,b)$, where $\SO(W,b)$ is the special orthogonal group of the form $b$, with the special unitary group $\SU(W, \psi)$ of the Hermitian form $\psi$, which is also an algebraic group defined over $E_0$. The operator $J$ commutes with the $G_0$-action on $W$.  A straightforward argument shows:

\begin{proposition}\label{standardrep} The algebra $\End_{E_0[G_0]}\Res_{E/E_0}W$ is equal to $E[J]$. Moreover, the representation $W$ of $G_0$ can not be defined over $E_0$ $\iff$ $E[J]$ is a division algebra $\iff$ $\lambda$ is not of the form $\Nm_{E/E_0}(c)$ for some $c\in E$. Thus,  the representation $W$ of $G_0$ can be defined over $E_0$ $\iff$ $E[J]$ is isomorphic to a matrix  algebra over $E_0$  $\iff$ $\lambda=\Nm_{E/E_0}(c)$ for some $c\in E$. 
\end{proposition}
\begin{proof} As an $E[G_0]$-module, $\Res_{E/E_0}W\otimes_{E_0}E = W \oplus \overline{W}$, where $\overline{W}$ is the conjugate vector space. The form $\psi$ defines an isomorphism from $W$ to $\overline{W}\spcheck$ and $b$ defines an isomorphism from $W$ to $W\spcheck$, hence $\psi$ and $b$ together define an isomorphism from $\overline{W}$ to $W$ as $E[G]$-modules. Thus, $\End_{E[G_0]}(\Res_{E/E_0}W \otimes _{E_0}E) \cong \mathbb{M}_2(E)$. It follows that $\dim_{E_0}\End_{E_0[G_0]}\Res_{E/E_0}W=4$, hence  is equal to $E[J]$ (note that $J \notin E$ since $J$ is conjugate linear). Also, $\End_{E_0[G_0]}\Res_{E/E_0}W$ is a matrix algebra if $\Res_{E/E_0}W$ is reducible and is a division algebra if $\Res_{E/E_0}W$ is irreducible. An argument as in  (3.22) of \cite{FL} shows  that $E[J]$ is a division algebra $\iff$ $\lambda$ is not a norm: every element of $E[J]$ can be uniquely written as $\alpha + \beta\cdot J$,    $\alpha, \beta\in E$. Using 
$$(\alpha + \beta\cdot J)(\bar{\alpha} -  \beta\cdot J) = \alpha\bar{\alpha} - \lambda\beta\bar{\beta},$$
we see that a nonzero $\alpha + \beta\cdot J$ is always invertible $\iff$ $\lambda$ is not a norm. Thus $E[J]$ is a division algebra $\iff$ $\lambda$ is not a norm, and hence is is isomorphic to a matrix  algebra over $E_0$  $\iff$ $\lambda$ is a norm.  
\end{proof}

\begin{remark} The proof of Proposition~\ref{standardrep} doesn't require that we know that $\overline{W}\cong W$ as $E[G]$-modules. In fact, if this were not the case, then we would have $\End_{E [G_0]}(\Res_{E/E_0}W\otimes _{E_0}E) \cong E\oplus E$, which is commutative of dimension $2$, but  in our situation $\dim_{E_0}(\End_{E_0[G_0]}\Res_{E/E_0}W) \geq 4$.
\end{remark}

In Section 3, we will need to know the Lie algebra $\mathfrak{g}_0$ of $G_0$ (as an $E_0$-vector space). As is well-known, the Lie algebra $\mathfrak{so}(W,b)$ is identified with $\bigwedge^2W$, by identifying $v\wedge w$ with the linear map $x\mapsto b(v,x)w - b(w,x)v$. In particular, a basis for $\mathfrak{so}(W,b)$ is given by $X_{rs}$, the linear map of $W$ corresponding to $e_r\wedge e_s$, $r< s$. Hence a basis is given by $X_{ij}$ and $X_{n+i, n+j}$ for $1\leq i < j\leq n$ as well as $X_{i, n+j}$ for $i,j \leq n$. The condition that a $T\in \mathbb{M}_{2n}(E_0)$ preserves the hermitian form $\psi$ is just the condition that $\psi(Tz, w) =- \psi(z, Tw) = -\overline{\psi(Tw, z)}$ for all $z, w\in W$. In terms of the basis $e_1, \dots, e_{2n}$ in which $\psi$ is diagonalized, these conditions read:   $T$  preserves $\psi$ $\iff$ $\psi(Te_r, e_s) = -\overline{\psi(Te_s, e_r)}$ for all $1\leq r,s \leq 2n$. Then a tedious calculation gives:

\begin{lemma}\label{Liealg} Suppose that $b$  and $\psi$ are compatible, and let $\alpha \in E_0$ be such that $\bar{\alpha} = -\alpha$. Then an $E_0$-basis for $\mathfrak{g}_0$ is given by:
\begin{align*}
a_{n+i}X_{ij} + a_jX_{n+i, n+j}, \qquad &1\leq i < j \leq n;\\
\alpha(a_{n+i}X_{ij} - a_jX_{n+i, n+j}), \qquad &1\leq i < j \leq n;\\
a_{n+i}X_{i, n+j} - a_{n+j}X_{j, n+i}, \qquad &1\leq i<j\leq n;\\
\alpha(a_{n+i}X_{i, n+j} + a_{n+j}X_{j, n+i}), \qquad &1\leq i<j\leq n;\\
\alpha X_{i,n+i}, \qquad &1\leq i  \leq n. \qquad \qed
\end{align*}
\end{lemma}

\section{Construction of the groups}

We keep the convention that $E_0$ is a totally real number field, resp.\  $E_0=\R$, and $E$ is an imaginary quadratic extension of $E_0$, resp.\  $E=\Cee$. 

\begin{definition}\label{defpsi} Let $e_1, \dots, e_{2n}$ be the standard $E$-basis of $W=E^{2n}$ and let $z = \sum_iz_ie_i$, $w = \sum_iw_ie_i$. For $\delta \in E_0$ and an integer $k$, $1\leq k \leq n$, define the Hermitian form $\psi_{\delta, k}$ by:
\begin{align*}
\psi_{\delta, k}(z,w) &= \delta z_1\bar{w}_1 + \dots + \delta z_k\bar{w}_k + z_{k+1}\bar{w}_{k+1} + \dots + z_n\bar{w}_n \\
&- z_{n+1}\bar{w}_{n+1} - \dots - z_{n+k}\bar{w}_{n+k}- \delta z_{n+k+1}\bar{w}_{n+k+1} - \dots - \delta z_{2n}\bar{w}_{2n}.
\end{align*}
In other words, in the previous notation, $a_i = \delta$ for $1\leq i \leq k$, $a_i = 1$ for $k+1 \leq i \leq n$, $a_{n+i} = -1$ for $1\leq i \leq k$, and $a_{n+i} = -\delta$ for $k+1 \leq i \leq n$. Hence $a_ia_{n+i} = \lambda = -\delta$ for all $i$, $1\leq i\leq n$, and so $J^2 = -\delta \Id$.
\end{definition} 

With the above definitions, $e_1, \dots, e_{2n}$ is a good basis, with $\lambda = -\delta$, and $W_1 = \operatorname{span}\{e_1, \dots, e_n\}$ is a good isotropic subspace of $W=E^{2n}$. We can restrict $\psi_{\delta, k}$ to the span $W_1$ of $e_1, \dots, e_n$, and in this case $\det \psi_{\delta, k}|W_1 = \delta^k$. On the other hand, there are some  permutations of  the good basis $e_1, \dots, e_n, e_{n+1}, \dots, e_{2n}$ for which it remains a good basis. Recall that the conditions we need are: (i) $b(e_i, e_{n+i}) = 1$ and all other $b(e_i, e_j) = 0$; (ii)  $\psi_{\delta, k}$ is diagonalized in the basis $e_1, \dots, e_n, e_{n+1}, \dots, e_{2n}$, say 
$\psi(z,w) = \sum_{i=1}^{2n}a_iz_i\bar{w}_i$; and finally 
(iii)  $a_ia_{n+i}$ is independent of $i$. Given integers $a,r$ with  $0\leq a\leq k$ and $0\leq r \leq n-k$, choose the new good basis $e_1', \dots, e_n', e_{n+1}', \dots, e_{2n}'$ as follows: 
\begin{enumerate}
\item For $i\leq a$, set $e_i' = e_i$ and $e_{n+i}' = e_{n+i}$; here $\psi(e_i', e_i') = \delta$ and $\psi(e_{n+i}', e_{n+i}') = -1$.
\item For $a+1 \leq i \leq k$, set $e_i' = e_{n+i}$ and $e_{n+i}'=e_i$; here $\psi(e_i', e_i') = -1$ and $\psi(e_{n+i}', e_{n+i}') = \delta$. 
\item For $k+1 \leq i \leq k+r$, set $e_i' = e_i$ and $e_{n+i}'=e_{n+i}$; here $\psi(e_i', e_i') = 1$ and $\psi(e_{n+i}', e_{n+i}') = -\delta$.
\item For $k+r+1 \leq i \leq n$, set $e_i' = e_{n+i}$ and $e_{n+i}'=e_i$; here $\psi(e_i', e_i') = -\delta$ and $\psi(e_{n+i}', e_{n+i}') = 1$.
\end{enumerate}
In other words, given $0\leq a\leq k$ and $0\leq r \leq n-k$,   after permuting the good basis $\{e_i\}$ by interchanging  $n-a-r$ of the $e_i$ with $e_{n+i}$, we arrive at another good basis where, for $i\leq n$,  $a$ of the coefficients of the $z_i\bar{w}_i$ are  $\delta$,  $k-a$ of the coefficients are $-1$, $r$ of the coefficients are $1$, and $n-k-r$ of the coefficients are $-\delta$. Then, defining $W_1' = \operatorname{span}\{e_1', \dots, e_n'\}$, $W_1'$ is a good isotropic subspace of $W$, with $\lambda = -\delta$ as before. Hence:

\begin{lemma}\label{chooseW1} Given $0\leq a\leq k$ and $0\leq r \leq n-k$, with $t = k-a$ and $s = n-k -r$, there exists a good isotropic subspace $W_1'$ of $W=E^{2n}$, with $\lambda =-\delta$, such that the determinant $\det (\psi_{\delta, k}|W_1') = \delta^a(-1)^t(-\delta)^s= (-1)^{t+s}\delta^{a+s}= (-1)^k\lambda^{a+s}$. In particular, if $a+s=2N$ is even, then $t+s \equiv k \pmod 2$, hence $\det (\psi_{\delta, k}|W_1') = (-1)^k\delta^{2N}$. \qed
\end{lemma}

\begin{definition} Let $G_{\delta, k}$ be the group of $E$-linear isomorphisms from $W$ to $W$ which have determinant $1$ and preserve the forms $b$ and $\psi_{\delta, k}$. Note that $G_{\delta, k}$, as a special case of the groups $G(W,b,\psi)$ defined in the last section,   is an algebraic group defined over $E_0$.
\end{definition}

We now classify the Hermitian forms $\psi_{\delta, k}$ and the groups $G_{\delta, k}$ in case $E_0 = \R$. Let $\psi_0=\psi_{1,k}$ (for any $k$) be the standard Hermitian form of signature $(n,n)$:
$$\psi_0(z,w) = \sum_{i=1}^n z_i\bar{w}_i - \sum_{i=1}^n z_{n+i}\bar{w}_{n+i}.$$ 
Let $\psi_{-1, k}$ be the Hermitian form of signature $(2n-2k, 2k)$ given by

\begin{align*}
\psi_{-1, k}(z,w) &= - z_1\bar{w}_1 - \dots - z_k\bar{w}_k + z_{k+1}\bar{w}_{k+1} + \dots + z_n\bar{w}_n \\
&- z_{n+1}\bar{w}_{n+1} - \dots - z_{n+k}\bar{w}_{n+k}+ z_{n+k+1}\bar{w}_{n+k+1} + \dots + z_{2n}\bar{w}_{2n}.
\end{align*}

Hence, for $\psi_0$, $\delta = 1$ and $\lambda = -1$, whereas,  for $\psi_{-1, k}$,  $\delta = -1$ and $\lambda = 1$. 

\begin{lemma} Suppose that $E_0 =\R$. Define the basis $e_i'$ of  $W$ by: $e_i ' = |\delta|^{-1/4}e_i$ if $1\leq i \leq k$ or $n+k+1\leq i \leq 2n$, and $e_i ' = |\delta|^{1/4}e_i$ if $k\leq i \leq n+k$. Then $b(e_i', e_j') = 0$ if $i\leq n$ and $j\neq n+i$, and $b(e_i', e_{n+i}') = 1$. Moreover:
\begin{enumerate}
\item[\rm(i)]   If $E_0 = \R$ and $\delta > 0$, then $\psi_{\delta, k}(e_i', e_j') = \delta^{1/2}  \psi_0(e_i, e_j)$. Hence there is a transformation of $W$ preserving $b$ which takes $\psi_{\delta, k}$ to the standard Hermitian form $\delta^{1/2}\psi_0$.
\item[\rm(ii)] If $E_0 = \R$ and $\delta <  0$, then $\psi_{\delta, k}(e_i', e_j') = |\delta|^{1/2}  \psi_{-1, k}(e_i, e_j)$. Hence there is a transformation of $W$ preserving $b$ which takes $\psi_{\delta, k}$ to the standard Hermitian form $|\delta|^{1/2}\psi_{-1, k}$. \qed
\end{enumerate}
\end{lemma}

\begin{corollary}  If $E_0 =\R$ and $\delta > 0$, then $G_{\delta, k} \cong \SO^*(2n)$. Moreover, $\operatorname{Res}_{\Cee/\R}W$ is an irreducible $\R[G]$-module, and its endomorphism algebra is $\mathbb{H}$.
\end{corollary}
\begin{proof}  By (i) of the lemma, we may replace $\psi_{\delta, k}$ with the standard form $\psi_0$. The last statement then follows from Proposition~\ref{standardrep}.
\end{proof}

\begin{corollary} If $E_0 =\R$ and $\delta > 0$, or if $E_0$ is a totally real number field such that there exists an embedding of $E_0$ in $\R$ for which $\delta > 0$ then the standard representation $W$ of $G_{\delta, k}$, which is defined over the imaginary quadratic extension $E$ of $E_0$, cannot be defined over $E_0$. \qed
\end{corollary}

To handle the case $E_0 =\R$ and $\delta < 0$, we may assume that $\delta = -1$. In this case, $J^2 = \Id$, and hence 
$\operatorname{Res}_{\Cee/\R}W$ is a direct sum of the $+1$ and $-1$ eigenspaces $W(+1)$ and $W(-1)$ of $J$. It is easy to see that multiplication by $\sqrt{-1}$ exchanges  $W(+1)$ and $W(-1)$, so that it is enough to consider just one of them, say $W(+1)$. 

\begin{lemma} For $E_0 = \R$ and $\delta < 0$, $G= G_{\delta, k} \cong G_{-1, k}\cong \SO(2n-2k, 2k)$, and the isomorphic real representations $W(\pm 1)$ are both isomorphic to the standard real representation  of $\SO(2n-2k, 2k)$.
\end{lemma}
\begin{proof} Clearly, $g\in G_{-1, k}$ preserves any two of $b$, $\psi_{-1, k}$, $J$ $\iff$ it preserves all three, and a complex linear $g$ commuting with $J$ defines a real linear transformation of $W(+1)$; conversely any real linear transformation of $W(+1)$ preserving the restriction of $b$ extends uniquely to a complex linear transformation of $W$ preserving $b$, $J$ and hence $\psi_{-1, k}$. So it will suffice to find a real basis of $W(+1)$ which diagonalizes $b| W(+1)$ and count the number of positive and negative eigenvalues.

Note that $J \colon W \to W$ is conjugate linear and satisfies: $J(e_i) = - e_{n+i}$, for $1\leq i\leq k$, $J(e_i) = e_{n+i}$, $k+1\leq i \leq n$, $J(e_{n+i}) = -e_i$, $1\leq i\leq k$, and $J(e_{n+i}) =  e_i$ for $k+1\leq i \leq n$. Then a real basis for the $2n$-dimensional real vector space $W(+1)$ is given by: $e_i + e_{n+i}$, $k+1\leq i \leq n$, $\sqrt{-1}(e_i - e_{n+i})$, $k+1\leq i \leq n$, $\sqrt{-1}(e_i + e_{n+i})$, $1\leq i\leq k$, and $e_i - e_{n+i}$, $1\leq i\leq k$. It is easy to check that this is a diagonal basis. Moreover, for $k+1\leq i \leq n$,
$$b(e_i + e_{n+i}, e_i + e_{n+i}) = b(\sqrt{-1}(e_i - e_{n+i}), \sqrt{-1}(e_i - e_{n+i})) = 2,$$
whereas for $1\leq i\leq k$ we have
$$b(\sqrt{-1}(e_i + e_{n+i}), \sqrt{-1}(e_i + e_{n+i})) = b(e_i - e_{n+i}, e_i - e_{n+i}) = -2.$$
Hence the signature of $b|W(+1)$ is $(2n-2k, 2k)$ as claimed.
\end{proof}

\begin{corollary} Let $E_0$ be a totally real number field, with $\sigma_1, \dots, \sigma_d$ the distinct embeddings of $E_0$ in $\R$, and let $\delta \in E_0$ be such that $\sigma_1(\delta) >0$ and $\sigma_i (\delta) < 0$ for $i>1$ (note that such $\delta$ exist by the approximation theorem). For $n = 6$ and $k=1$, define the group $G_{\delta, 1}$ as above and let $G_{\delta, 1; i, \R}$ be the real group defined by extension of scalars via the embedding $\sigma_i$. Then $G_{\delta, 1; 1, \R} \cong \SO^*(12)$ and $G_{\delta, 1; i, \R} \cong \SO(2, 10)$ for $i > 1$. \qed
\end{corollary}

\section{Rationality of the half spin representation}

In the preceding section, we constructed a group $G_{\delta, 1}$ defined over the totally real number field $E_0$, such that, for a suitable $\delta$,  the real group induced by the embedding $\sigma_1$ is isomorphic to $\SO^*(12)$ and the real groups induced by the embedding $\sigma_i$, $i> 1$, are isomorphic to $\SO(2,10)$. Viewing $G_{\delta, 1}$ as a subgroup of $\SO(W,b)$, the orthogonal group associated to $b$, let $G_1$ be the inverse image of $G_{\delta, 1}$  in the Spin double cover $\Spin(W,b)$ of $\SO(W,b)$. Hence $G_1$ acts on the half spin representations $S^{\pm}$ of $\Spin(W,b)$  associated to the $E$-vector space $W$, where $W$ is the standard representation of $\SO(W,b)$ and of $G_{\delta, 1}$. The representations $S^{\pm}$ of $G_1$  are  \emph{a priori} only defined over $E$, and we want to find necessary and sufficient conditions for them to be defined over $E_0$. As in Proposition~\ref{standardrep}, this amounts to deciding when  the endomorphism algebra  of the $E_0[G_1]$-module $\Res_{E/E_0}S^{\pm}$ is a matrix algebra $\mathbb{M}_2(E_0)$ and when it is a division algebra.  

We begin by reviewing the salient properties of the Clifford algebra  $C=C(W,b)$ of $b$  and the half spin representations, using Fulton--Harris \cite{fultonharris} as a reference (but our notation differs slightly from theirs) or Bourbaki \cite{bourbaki2}. In particular, for a fixed good  basis $e_1, \dots, e_{2n}$ of $W$, $C(W,b)$ is the $\Zee/2\Zee$-graded quotient of the tensor algebra $T^*W$ by the relations $e_ie_j = -e_je_i$, $j \neq n \pm i$, and $e_ie_{n+i} + e_{n+i}e_i = 2$. (Here and in the rest of this paper the tensor and exterior algebras of $W$ or $W_1$ will always be as $E$-vector spaces.)

Let $W_1$ be the good isotropic $E$-vector subspace of $W$ spanned by $e_1, \dots, e_n$ and let $W_2$ be the $b$-isotropic $E$-vector subspace of $W$ spanned by $e_{n+1}, \dots, e_{2n}$. 
There is an algebra isomorphism  $C(W,b) \cong \End(\bigwedge^{\bullet} W_1)$ and a corresponding isomorphism  $C^{\text{\rm{even}}}(W,b) \cong   \End(\bigwedge^{\text{\rm{even}}}W_1)\oplus \End(\bigwedge^{\text{\rm{odd}}}W_1)$: Given $e_i\in W_1$, $e_i$ acts on $\bigwedge^{\bullet} W_1$ via $\ell(e_i)$, wedge product with $e_i$, and $e_{n+i}\in W_2$ acts via $2\iota(e_{n+i})$, the interior product with $e_{n+i}$ viewed as an element of $W_1\spcheck$. However, to avoid the annoying factors of $2$ in what follows, we will scale $b$ by $1/2$. This does not change any of the calculations in an essential way but replaces  the defining relation $e_ie_{n+i} + e_{n+i}e_i = 2$ by $e_ie_{n+i} + e_{n+i}e_i = 1$ and hence   $e_{n+i}\in W_2$  now acts via $\iota(e_{n+i})$, i.e.\ without the factor of $2$. The half spin representation spaces  of $\Spin(W,b)$ are then given by
$$
S^+  = \bigwedge^{\text{\rm{even}}}W_1;\qquad 
S^-  = \bigwedge^{\text{\rm{odd}}}W_1.
$$
Here the action of the Lie algebra $\mathfrak{so}(W,b)$ on $S^\pm$ is given  in \cite[p. 305]{fultonharris}, and will be recalled later.
In terms of the Bourbaki labeling of the simple roots \cite{bourbaki}, $S^+$ has highest weight $\varpi_{n}$ and $S^-$ has highest weight $\varpi_{n-1}$ in case $n =2m$ is even (the only case which will concern us), whereas $S^+$ has highest weight $\varpi_{n-1}$ and $S^-$ has highest weight $\varpi_{n}$ in case $n$ is odd. Moreover, for $n$ even there are nondegenerate $\Spin(W,b)$-invariant forms on $S^\pm$ (which are either symmetric or symplectic depending on the parity of $m$), and hence $(S^+)\spcheck \cong S^+$, $(S^-)\spcheck \cong S^-$ as $E[\Spin(W,b)]$-modules.

The direct sum decomposition $W = W_1\oplus W_2$ induces a $\Zee$-grading on $C(W,b)$ and on $C^{\text{\rm{even}}}(W,b)$, where the elements in $W_1$ have degree $1$ and those in $W_2$ have degree $-1$, since the only interesting  relation is $e_ie_{n+i} + e_{n+i}e_i = 1$ which has degree zero (the relations $e_i e_j = -e_je_i$, $j \neq i\pm n$, don't cause a problem). We write 
$$C(W,b)= \bigoplus_{d\in \Zee}C_d(W,b) \text{ and } C^{\text{\rm{even}}}(W,b) = \bigoplus _{d\in \Zee}C^{\text{\rm{even}}}_{2d}(W,b).$$
The gradings on $C(W,b)$ and on $C^{\text{\rm{even}}}(W,b)$ correspond to the natural gradings:
\begin{gather*}
\End(\bigwedge^{\bullet} W_1) \cong \bigoplus_{d\in \Zee}\End^d(\bigwedge^{\bullet} W_1)\\
\End(\bigwedge^{\text{\rm{even}}} W_1) \cong \bigoplus_{d\in \Zee}\End^{2d}(\bigwedge^{\text{\rm{even}}} W_1), \qquad \End(\bigwedge^{\text{\rm{odd}}} W_1) \cong \bigoplus_{d\in \Zee}\End^{2d}(\bigwedge^{\text{\rm{odd}}} W_1),
\end{gather*}
 where  
 $A\in \End^d(\bigwedge^{\bullet} W_1)$ $\iff$ $ A(\bigwedge^kW_1) \subseteq \bigwedge^{k+d}W_1$, and similarly for the summands $\End(\bigwedge^{\text{\rm{even}}} W_1)$, $\End(\bigwedge^{\text{\rm{odd}}} W_1)$.
 
We now describe how the half spin representations depend on the choice of the isotropic subspace $W_1$. In what follows, we use freely the notation of \cite{bourbaki}. Fix once and for all a good  basis $\mathbf{e} = e_1, \dots, e_{2n}$, and hence isotropic subspaces  $W_1 = \operatorname{span}\{e_1, \dots, e_n\}$ and $W_2 = \operatorname{span}\{e_{n+1}, \dots, e_{2n}\}$. Then there is a natural choice of maximal torus $T$ corresponding to diagonal matrices in the basis $\mathbf{e}$, hence weights $\varepsilon_i$ for $1\leq i\leq n$. Note that, in an obvious sense, $\varepsilon_{n+i} = -\varepsilon_i$, since if $g\in T$ and $g(e_i) = ce_i$, then $g(e_{n+i}) = c^{-1}e_{n+i}$. The  half spin representations $S^\pm$ depend on the choice of the isotropic subspace $W_1$ and hence on the good   basis $\mathbf{e} = e_1, \dots, e_{2n}$. To emphasize this dependence, we shall write $S^\pm(\mathbf{e})$ where necessary.  

Now suppose, as in Lemma~\ref{chooseW1},  that we are given another good  basis $\mathbf{e}' = e_1', \dots, e_{2n}'$, which is obtained from $\mathbf{e}$ by switching $k$ of the basis vectors $e_i$, $i\leq n$, with the corresponding basis vectors $e_{n+i}$. In other words, there exists a subset $I\subseteq \{1, \dots, n\}$ with $\#(I) = k$ such that, for $i\in I$,  $e_i'= e_{n+i}$ and $e_{n+i}' = e_i$, whereas, for $i\notin I$, $e_i'=  e_i$ and $e_{n+i}' =e_{n+i}$.  We obtain isotropic subspaces $W_1'=\operatorname{span}\{e_1', \dots, e_n'\}$ and $W_2' = \operatorname{span}\{e_{n+1}', \dots, e_{2n}'\}$ as before.  Let $S^\pm(\mathbf{e}')$ be the half spin representations constructed using the isotropic subspaces $W_1'$, $W_2'$. Then we have the following:

\begin{lemma}\label{newhalfspins} In the notation above, if $k$ is even, then, as $\Spin(2n)$-modules, $S^+(\mathbf{e}')\cong  S^+(\mathbf{e})$ and $S^-(\mathbf{e}')\cong  S^-(\mathbf{e})$.   If $k$ is odd, then  $S^+(\mathbf{e}')\cong  S^-(\mathbf{e})$ and $S^-(\mathbf{e}')\cong  S^+(\mathbf{e})$.  
\end{lemma}
\begin{proof}
Let $\sX(S^\pm(\mathbf{e}))$ denote  the set of weights for $S^\pm$, and similarly for $\sX(S^\pm(\mathbf{e}'))$.  Clearly $\sX(S^\pm(\mathbf{e}')) = \varphi(\sX(S^\pm(\mathbf{e})))$, where $\varphi$ is the isometry of the weight lattice given by switching $\varepsilon_i$ to $-\varepsilon_i$ for $i\in I$, with a total of $k$ sign changes. Hence, if   $W$ is the Weyl group of $D_n$, then  $\varphi\in W$   $\iff$ $k$ is even. Moreover, if $\varphi\notin W$, then the automorphism of the Dynkin diagram corresponding to $\varphi$ exchanges $\alpha_{n-1}$ and $\alpha_n$, hence $\varpi_{n-1}$ and $\varpi_n$. It follows that, if $k$ is even, then $\sX(S^\pm(\mathbf{e}')) = \sX(S^\pm(\mathbf{e}))$ and hence $S^\pm(\mathbf{e}') \cong S^\pm(\mathbf{e})$, whereas if $k$ is odd then $\sX(S^\pm(\mathbf{e}')) = \sX(S^\mp(\mathbf{e}))$ and hence  $S^\pm(\mathbf{e}') \cong S^\mp(\mathbf{e})$. 
\end{proof}

We return  to the general situation of a compatible $E$-bilinear form $b$ and an $(E, E_0)$-Hermitian form $\psi$, with notation as in Section 1, Definition~\ref{goodcomp}.  

\begin{lemma} Let $G_1$ be the neutral component of the preimage  in $\Spin(W,b)$ of the group $G(W, b, \psi)$ which stabilizes both $b$ and $\psi$ in the group of units of $C^{\text{\rm{even}}}(W,b)$. Then $G_1$ is an algebraic group defined over $E_0$.
\end{lemma}
\begin{proof} Let   $g\in \Spin(W,b)$. Then  $gWg^{-1} = W$, and, if we define $\rho(g) = gwg^{-1}$, then $\rho$ is the double cover homomorphism from $\Spin(W,b)$ to $\SO(W,b)$. There is thus an induced homomorphism, also denoted by $\rho$, from $\Res_{E/E_0}\Spin(W,b)$ to $\Res_{E/E_0}\SO(W,b)$.
The argument  that $G_1$ is defined over $E_0$ is then similar to the discussion in Section 1 for the group $G(W, b, \psi)$: By definition, $G_1$ is the inverse image $\rho^{-1}(\SU(W, \psi))$ in $\Res_{E/E_0}\Spin(W,b)$ of  $\SU(W, \psi)$, which is defined over $E_0$. Hence $G_1$ is defined over $E_0$.
\end{proof} 

With this said, we can now state the main theorem of this section as follows:

\begin{theorem}\label{allcases} Suppose that $n = 2m$ is even. Let $b$ and $\psi$ be compatible, let $W_1$ be a good isotropic subspace of $W$, and let  $\lambda= a_ia_{n+i}$ be as in Lemma \ref{lemlambda} and $D  = \det (\psi|W_1)= a_1 \cdots a_n$ be as in Definition~\ref{goodcomp}. Then:
\begin{enumerate}
\item[\rm(i)] There exists a conjugate linear operator $L_+ \in \End_{E_0[G_1]}\Res_{E/E_0}S^+$ such that $(L_+)^2 = (-1)^mD\Id$ if $m$ is even and $(L_+)^2 = (-1)^mD\lambda\Id$  if $m$ is odd.
\item[\rm(ii)] There exists a conjugate linear operator $L_- \in \End_{E_0[G_1]}\Res_{E/E_0}S^-$ such that $(L_-)^2 = (-1)^mD\lambda\Id$ if $m$ is even and $(L_-)^2 = (-1)^mD\Id$ if $m$ is odd.
\end{enumerate} 
\end{theorem}

Using the theorem, we can completely describe when the representations $S^\pm$ are defined over $E_0$:

\begin{corollary}\label{criterion} With  notation and hypotheses as above,
\begin{enumerate}
\item[\rm(i)] Both of the representations $S^+$ and $S^-$ can be defined over $E_0$ $\iff$  $(-1)^mD$  and  $(-1)^mD\lambda$ are norms, i.e.\ lie  in $\Nm_{E/E_0}(E^*) \subseteq E_0^*$.
\item[\rm(ii)] The representation  $S^+$ can be defined over $E_0$ and  $S^-$ cannot be defined over $E_0$ $\iff$ either $m$ is even,   $(-1)^mD$   is a  norm and $(-1)^mD\lambda$ is not a norm, or $m$ is odd,  $(-1)^mD$ is not a norm and  $(-1)^mD\lambda$ is a  norm.
\item[\rm(iii)] The representation  $S^-$ can be defined over $E_0$ and  $S^+$ cannot be defined over $E_0$ $\iff$ either $m$ is odd,   $(-1)^mD$   is a  norm and $(-1)^mD\lambda$ is not a norm, or $m$ is even,  $(-1)^mD$ is not a norm and  $(-1)^mD\lambda$ is a  norm.
\item[\rm(iv)] Neither of the representations  $S^+$, $S^-$ can be defined over $E_0$ $\iff$ neither $(-1)^mD$  nor  $(-1)^mD\lambda$ are norms.
\end{enumerate}
\end{corollary}
\begin{proof} The argument of Proposition~\ref{standardrep} shows that $\End_{E_0[G_1]}\Res_{E/E_0}S^+=E[L_+]$ and that $\End_{E_0[G_1]}\Res_{E/E_0}S^-=E[L_-]$, and moreover that $S^\pm$ can be defined over $E_0$ $\iff$ $(L_\pm)^2 =c_\pm\Id$, where $c_\pm\in E_0$ is a norm. The various cases of the corollary then follow from the cases in Theorem~\ref{allcases}.
\end{proof}

For the next two corollaries, we shall apply Theorem~\ref{allcases} to the forms $\psi_{\delta, k}$ of the preceding section. In this case, there is a fixed good basis $\mathbf{e} = e_1, \dots, e_{2n}$ for which the form $\psi_{\delta, k}$ is given by Definition~\ref{defpsi}.  The half spin representations $S^\pm$ will always mean the representations $S^\pm(\mathbf{e})$ with respect to this basis. Then we have the following, which includes the final piece of Theorem~\ref{mainthm2}, the rationality statement for the representation $S^+$:

\begin{corollary}\label{usedcor} Suppose that $G_1 = \widetilde{G}_{\delta, k}$ is the spin double cover of the group $G_{\delta, k}$ corresponding to the form $\psi = \psi_{\delta, k}$ and that $n = 2m$ with $m\equiv k \pmod 2$. Then the $E[G_1]$-module $S^+$ is the extension to $E$ of an $E_0[G_1]$-module.  

In particular, if $k=1$ and   $n= 2m = 6$, then the representation $S^+$ of $G_1 = \widetilde{G}_{\delta, 1}$ is defined over $E_0$.
\end{corollary}
\begin{proof} For the form $\psi_{\delta, k}$, by a  choice of good basis $\mathbf{e}'$ and isotropic subspace $W_1'$ as in Lemma~\ref{chooseW1}, we can assume that $D = (-1)^k\delta^{2N}$. Since $m\equiv k \pmod 2$, 
$$(-1)^mD = (-1)^m(-1)^k\delta^{2N} = \delta^{2N}$$ is a square and hence a norm. Thus,   for $m$ even, $S^+(\mathbf{e}')$ can be defined over $E_0$, and,  for $m$ odd, $S^-(\mathbf{e}')$ can be defined over $E_0$. In the notation of the discussion prior to Lemma~\ref{chooseW1}, the total number of basis vectors switched is $t+s$, and moreover $t+s \equiv k\equiv m\pmod 2$. Hence, by Lemma~\ref{newhalfspins}, if $m$ is even then $S^+(\mathbf{e}') \cong  S^+(\mathbf{e})$ and if $m$ is odd then $S^-(\mathbf{e}') \cong  S^+(\mathbf{e})$. In all cases $S^+(\mathbf{e}) = S^+$ is defined over $E_0$.
\end{proof}

 In case $E_0 =\R$, we have the following, which is a special case of the criterion of \cite[Theorem IV.E.4]{mtbook} (see also \cite[(26.27)]{fultonharris} for the cases $k=0, n$ in (ii)).   
 
\begin{corollary}\label{realcase} Let $E_0 =\R$ and let $n=2m$.
\begin{enumerate}
\item[\rm(i)] For the spin double cover of $\SO^*(4m)$,   $S^+$ can be defined over $\R$ and $S^-$ cannot be defined over $\R$.  
\item[\rm(ii)] For the spin double cover of $\SO(2n-2k,2k)$,   if $k\equiv m \pmod 2$, then both $S^+$ and $S^-$ as well as the standard representation are defined over $\R$, while if $k \not\equiv m \pmod 2$, then neither $S^+$ nor $S^-$ is defined over $\R$.
\end{enumerate}
\end{corollary}
\begin{proof} Both cases are reduced to Corollary~\ref{criterion}: (i) In this case, $\psi = \psi_0$ and   $\lambda = -1$. Beginning with the good  basis $\mathbf{e}=e_1, \dots, e_{2n}$, by Lemma~\ref{chooseW1}, if $\mathbf{e}'$ is the good  basis obtained by switching $m$ of the $e_i$ to $e_{n+i}$, then $D=(-1)^m$ in the new basis. If $m$ is even, by Lemma~\ref{newhalfspins},   $S^\pm(\mathbf{e}') = S^\pm(\mathbf{e})$, but for $m$ odd $S^\pm(\mathbf{e}') = S^\mp(\mathbf{e})$. By Corollary~\ref{criterion}, as $\lambda = -1$ and $(-1)^mD = 1$, if $m$ is even then $S^+(\mathbf{e}')$ can be defined over $\R$ and $S^-(\mathbf{e}')$ cannot be defined over $\R$, whereas if $m$ is odd then $S^-(\mathbf{e}')$ can be defined over $\R$ and $S^+(\mathbf{e}')$ cannot be defined over $\R$. Thus, in all cases,   $S^+(\mathbf{e})$ can be defined over $\R$ and $S^-(\mathbf{e})$ cannot be defined over $\R$. 

\smallskip
\noindent (ii) In this case, $\psi = \psi_{-1,k}$,  $\lambda = 1$, and $D = (-1)^k$, hence $(-1)^mD = (-1)^{k+m}$. Thus, if $k\equiv m \pmod 2$, then both $(-1)^mD$ and $(-1)^mD\lambda$ are equal to $1$, hence lie in $\Nm_{\Cee/\R}\Cee^*$, and so both $S^+$ and $S^-$ are defined over $\R$. If $k \not\equiv m \pmod 2$, then both $(-1)^mD$ and $(-1)^mD\lambda$ are equal to $-1$, and so neither $S^+$ nor $S^-$ is defined over $\R$.
\end{proof}

  In particular, even if the standard representation $W$ can be defined over $E_0$ (for which Proposition~\ref{standardrep} gives the necessary and sufficient condition that $\lambda$ is a norm), it is not always the case that the half spin representations can be defined over $E_0$. 
  
\begin{remark} The condition that $(-1)^mD$ is a norm, where $D = \det (\psi|W_1)$, is not intrinsically defined: it depends on the choice of a good isotropic subspace. For example, switching $e_1$ and $e_{n+1}$ as in the discussion before Lemma~\ref{chooseW1} replaces $D$ by $D\lambda$ mod $E_0^2$, and it is certainly possible that one of $(-1)^mD$, $(-1)^mD\lambda$ is a norm but the other is not. Thus the condition that $S^+$ can be defined over $E_0$ would appear to depend on the choice of $W_1$. However, the operation of switching $e_1$ and $e_{n+1}$ also switches  $S^+$ and  $S^-$, by Lemma~\ref{newhalfspins}, so the final result is in fact consistent. 
\end{remark}

\begin{proof}[Proof of Theorem~\ref{allcases}]
 We shall use the Hodge $\star$-operator associated to $\psi|W_1$ and the volume form $e_1\wedge \cdots \wedge e_n$ as defined in \cite{FL}, \S3.5, and its natural  generalization to forms of arbitrary degree.  In case $n=2m$ is even, $\star$ maps $\bigwedge^{\text{\rm{even}}} W_1$ to $\bigwedge^{\text{\rm{even}}} W_1$ and $\bigwedge^{\text{\rm{odd}}} W_1$ to $\bigwedge^{\text{\rm{odd}}} W_1$. The operator $\star$ is conjugate linear and hence $E_0$-linear, and, by an argument along the lines of Lemma 3.21 of \cite{FL}, $\star \star = (-1)^{k(n-k)}D \Id$ and hence is equal to $D \Id$ on forms of even degree and $-D \Id$ on forms of odd degree . Note however that there is no reason to expect that $\star$ will commute with the $G_1$-action induced by the inclusion $G_1\subseteq \Spin (W,b)$. Next we use:

\begin{lemma}\label{starcomps}  As operators from $\bigwedge ^kW_1$ to $\bigwedge ^{k\mp 1}W_1$,
\begin{align*}
\star\, \ell(e_i)\,\star &  = (-1)^{n(k+1)}Da_i\iota(e_{n+i});\\
\star\, \iota(e_{n+i})\, \star &  = (-1)^{nk+1}D\lambda^{-1}a_{n+i}\ell(e_i).
\end{align*}
\end{lemma}
\begin{proof} For $I =\{i_1, \dots, i_k\}$ a  subset of $\{1, \dots, n\}$ with $i_1< \cdots < i_k$ and using the shorthand $e_I$ for $e_{i_1}\wedge \cdots \wedge e_{i_k}$ and $a_I = a_{i_1}\cdots a_{i_k}$, we have $\star e_I= \varepsilon_{I,I'}a_Ie_{I'}$ as in \cite{FL}, \S3.5, where $I' =\{1,2,\dots, n\} -I$ is the complementary index set and $\varepsilon_{I,I'} =\pm 1$ is a sign factor only depending on $k$. Hence
$$\star(e_i\wedge \star e_I) = \begin{cases} 0, &\text{if $i\notin I$}; \\
\star( \varepsilon_{I,I'}a_Ie_i \wedge e_{I'}) = \pm Da_ie_{I -\{i\}},  &\text{if $i\in I$}. 
\end{cases}$$
  With some care as to the sign, one checks that
 $$\star\, \ell(e_i)\star   = (-1)^{n(k+1)}Da_i\iota(e_{n+i}).$$
 This proves the first equality in the lemma. To prove the second equality, use
 $$\star\, \iota(e_{n+i})\star = (-1)^{n(k+1)}D^{-1}a_i^{-1}\star \star\, \ell(e_i)\star \star $$
and the fact that $a_i^{-1} = \lambda^{-1}a_{n+i}$.
\end{proof}

\begin{corollary}\label{morestarcomps} With notation as above, as operators on $\bigwedge ^kW_1$,
\begin{align*}
\star \, \ell(e_i) & = (-1)^ka_i \iota(e_{n+i})\star; \\
\star \, \iota(e_{n+i}) & = (-1)^{k+n+1}\lambda^{-1}a_{n+i}\ell(e_i)\star. \hfill  \qed
\end{align*}
\end{corollary}

\begin{definition} Define   $L_+\colon \Res_{E/E_0}S^+\to \Res_{E/E_0}S^+$ and $L_-\colon \Res_{E/E_0}S^-\to \Res_{E/E_0}S^-$ as follows: in all cases, $L_\pm|\bigwedge ^kW_1 = (-1)^d\lambda^{-d}\,\star$, where $\D d= \left[\frac{k-m}{2}\right]$. Explicitly,  for $m$ even,  $L_+\colon \bigwedge ^{m+2d}W_1\to \bigwedge ^{m-2d}W_1$ is given by $(-1)^d\lambda^{-d}\,\star$ and $L_-\colon \bigwedge ^{m+2d+1}W_1\to \bigwedge ^{m-2d-1}W_1$ by $(-1)^d\lambda^{-d}\,\star$. For $m$ odd,   $L_+|\bigwedge ^{m+2d+1}W_1 = (-1)^d\lambda^{-d}\,\star$ and $L_-|\bigwedge ^{m+2d}W_1 = (-1)^d\lambda^{-d}\,\star$.
\end{definition}

We can now complete the proof of Theorem~\ref{allcases}. Clearly, the operators $L_{\pm}$ are conjugate linear. It is easy to check that, if $m$ is even,  $L_+^2 = D\Id = (-1)^mD\Id$ and $L_-^2 = D\lambda\Id = (-1)^mD\lambda\Id$, whereas if $m$ is odd then  $L_+^2 = -D\lambda\Id = (-1)^mD\lambda\Id$ and $L_-^2 = -D\Id = (-1)^mD\Id$. Finally we must show that  the $L_{\pm}$ commute with the action of $G_1$.  Since $G_1$ is connected, it suffices  to show that the $L_\pm$ commute with the action of $\mathfrak{g}_1 = \mathfrak{g}_0$, for which we have written down a basis in Lemma~\ref{Liealg}. By \cite{fultonharris}, for $r<s$, the element $X_{rs}\in \mathfrak{so}(W,b)$ corresponding to the element $e_r\wedge e_s\in \bigwedge^2W_1$ then corresponds to the element $e_re_s- b(e_r, e_s)$ of $C(W,b)$, and hence is equal to $e_re_s$ except for the case $r=i$, $s=n+i$, in which case it corresponds to the element $e_ie_{n+i} - \frac12$ (with our scaling conventions on $b$ in this section). Thus for example if $1\leq i< j\leq n$, $X_{ij}$ corresponds to $e_ie_j$, and hence via the isomorphism $C^{\text{\rm{even}}}(W,b) \cong   \End(\bigwedge^{\text{\rm{even}}}W_1)\oplus \End(\bigwedge^{\text{\rm{odd}}}W_1)$ to the operator $\ell(e_i)\ell(e_j)$, and similarly $X_{n+i, n+j}$ corresponds to   $\iota(e_{n+i})\iota(e_{n+j})$. Now a brute force computation  completes the proof. For example, for $1\leq i < j\leq n$, to see that   $L_\pm$ commutes with $T=a_{n+i}\ell(e_i)\ell(e_j) + a_j\iota(e_{n+i})\iota(e_{n+j})$, we  let both sides act on $\bigwedge^kW_1$. With $\D d = \left[\frac{k-m}{2}\right]$, we must compare the two expressions
\begin{align*}
TL_\pm   & = (-1)^d\lambda^{-d} a_{n+i}\ell(e_i)\ell(e_j)\star   + (-1)^d\lambda^{-d}a_j\iota(e_{n+i})\iota(e_{n+j})\star ; \\
L_\pm T  & = (-1)^{d+1}\lambda^{-d-1} a_{n+i}\star \ell(e_i)\ell(e_j)   + (-1)^{d-1}\lambda^{-d+1}a_j\star \iota(e_{n+i})\iota(e_{n+j})  .
\end{align*}
Using Corollary~\ref{morestarcomps}, it follows  that 
\begin{align*}
(-1)^{d-1}\lambda^{-d+1}a_j\star \iota(e_{n+i})\iota(e_{n+j}) &= (-1)^{d-1+k-1}\lambda^{-d+1}a_j\lambda^{-1}a_{n+i}\ell(e_i)\star \, \iota(e_{n+j})\\
&= (-1)^{d-1+k-1+k}\lambda^{-d}a_ja_{n+i}a_{n+j}\ell(e_i)\ell(e_j)\star\\
&= (-1)^d\lambda^{-d} a_{n+i}\ell(e_i)\ell(e_j)\star.
\end{align*}
The other terms are similar.
This concludes the proof of Theorem~\ref{allcases}.
\end{proof}

\begin{remark} We sketch another proof for the somewhat mysterious fact that the operators $L_\pm$ commute with the action of $G_1$, the main point of the proof of Theorem~\ref{allcases}.   Let $\Phi\colon W \to W'$ be an $E$-linear isomorphism and let $b_\Phi$ be the induced bilinear form on $W'$ defined in Section 1. Then there is an induced isomorphism of $E$-algebras $\Phi^*\colon C(W, b) \to C(W', b_\Phi)$. For example, the form $b$ induces an isomorphism $B^* \colon C(W, b) \to C(W\spcheck, b\spcheck)$. Given a conjugate linear isomorphism $\Psi\colon W \to W'$, there is the induced $E$-bilinear form $\bar{b}_\Psi$ defined in Section 1.\ Using $\Psi$ to define (naturally) a conjugate linear isomorphism from the tensor algebra of $W$ to the tensor algebra of $W'$, it is easy to see that $\Psi$ carries the defining relation $v\otimes v - b(v,v)\cdot 1$ to  $\Psi(v)\otimes \Psi(v) - \overline{b(v,v)}\cdot 1$, which is just the defining relation  $\Psi(v)\otimes \Psi(v) - \bar{b}_\Psi(\Psi(v),\Psi(v)) \cdot 1$. Hence $\Psi$ induces a conjugate linear isomorphism, denoted $\Psi^*$, from $C(W, b)$ to $C(W', \bar{b}_\Psi)$. Note that $\Psi^*$ as well as $\Phi^*$ preserve the decomposition into even and odd degrees. Finally, if $t\in E^*$, define $h_t \colon T^*W \to T^*W$ by: $h_t$ is multiplication by $t^{[k/2]}$ on the graded homogeneous piece $T^{k}W$.   Then $h_t \colon T^{\text{\rm{even}}}W \to T^{\text{\rm{even}}}W$ is an algebra isomorphism which descends to an algebra isomorphism (also denoted $h_t$) from $C^{\text{\rm{even}}}(W,b)$ to $C^{\text{\rm{even}}}(W,t^{-1}b)$. Moreover, the action of $h_t$ descends to an operator $h_t$ from from $C^{\text{\rm{odd}}}(W,b)$ to $C^{\text{\rm{odd}}}(W,t^{-1}b)$ which is compatible with its  structure as a module over $C^{\text{\rm{even}}}(W,b)$.

In particular, if   $b$ and $\psi$ are compatible forms as defined in Definition~\ref{goodcomp}, we have the $E$-linear algebra isomorphism $B^*\colon  C(W, b) \to C(W\spcheck, b\spcheck)$ and the conjugate linear algebra isomorphism $\Psi^* \colon C(W, b) \to C(W\spcheck, \lambda^{-1}b\spcheck)$. Combining, we have a conjugate linear isomorphism $J^* = (B^*)^{-1}\circ \Psi^*\colon C(W,b) \to C(W, \lambda^{-1}b)$, with $(J^*)^2 = \lambda\Id$. Composing with the operator $h_{\lambda^{-1}}$  gives a conjugate linear algebra isomorphism
$$\mathcal{L} = h_{\lambda^{-1}} \circ J^* = h_{\lambda^{-1}}\circ (B^*)^{-1}\circ \Psi^*\colon C^{\text{\rm{even}}}(W,b) \to C^{\text{\rm{even}}}(W,b).$$
Note that $\mathcal{L}$ sends $C^{\text{\rm{even}}}_{2d}(W,b)$ to $C^{\text{\rm{even}}}_{-2d}(W,b)$, and a calculation gives $\mathcal{L}^2 = \Id$, so that $\mathcal{L}$ is a conjugate linear involution of $C^{\text{\rm{even}}}(W,b)$. A straightforward argument shows that $\mathcal{L}(g) = g$ for all $g\in G_1$, hence $\mathcal{L}(g\xi) = g\mathcal{L}(\xi)$. Similarly, there is  a conjugate linear map $\mathcal{M} \colon C^{\text{odd}}(W,b) \to C^{\text{odd}}(W,b)$ of the form $h_{\lambda^{-1}}\circ J^*$. It satisfies:  $\mathcal{M}^2 = \lambda \Id$ and, for all $\xi \in C^{\text{even}}(W,b)$ and $\eta \in C^{\text{odd}}(W,b)$,  $\mathcal{M}(\xi\eta) = \mathcal{L}(\xi)\mathcal{M}(\eta)$. In particular, for $g\in G_1$,
$$\mathcal{M}(g\eta) = \mathcal{L}(g)\mathcal{M}(\eta) = g\mathcal{M}(\eta).$$
Via the algebra isomorphism $C^{\text{\rm{even}}}(W,b) \cong \mathcal{E} =\End(\bigwedge^{\text{\rm{even}}} W_1)\oplus \End(\bigwedge^{\text{\rm{odd}}} W_1)$,  view $\mathcal{L}$ as an involution on $\mathcal{E}$. Then a computation using Lemma~\ref{starcomps} shows that $\mathcal{L}$  preserves the two factors in the direct sum and can be computed as follows: For $A \in \End^{2d}(\bigwedge^{\text{\rm{even}}} W_1)$, 
$$\mathcal{L}(A) = (-1)^d\lambda^{-d}D^{-1}\star\,  A \, \, \star,$$
and for $A \in \End^{2d}(\bigwedge^{\text{\rm{odd}}} W_1)$, 
$$\mathcal{L}(A) = (-1)^{d+1}\lambda^{-d}D^{-1}\star\,  A \, \, \star.$$
There are similar formulas for $\mathcal{M}$.
 
There is no reason to expect $\mathcal{L}$  or $\mathcal{M}$ to induce operators on $S^\pm$ which commute with $G_1$. However, if $\alpha \in\bigwedge^mW_1$, say,  is nonzero, then $C^{\text{\rm{even}}}(W,b)\cdot \alpha = \mathcal{E}\cdot \alpha$ is equal to $S^+$ if $m$ is even and $S^-$ if $m$ is odd since $S^\pm$ is a simple $C^{\text{\rm{even}}}(W,b)$-module.  Similarly,  $C^{\text{\rm{odd}}}(W,b)\cdot \alpha=S^-$ if $m$ is even and $S^+$ if $m$ is odd. Hence, depending on the parity of $m$,  $S^\pm \cong \mathcal{E}/I_\alpha^{\text{\rm{even}}}$, where $I_\alpha^{\text{\rm{even}}}$ is  the left ideal in $\mathcal{E}$ corresponding to the kernel of  evaluation at the point $\alpha$. In particular, if $\mathcal{L}(I_\alpha^{\text{\rm{even}}}) = I_\alpha^{\text{\rm{even}}}$,   there is an induced  conjugate linear action of $\mathcal{L}$ on $S^\pm$ which commutes with the $G_1$-action and satisfies $\mathcal{L}^2 = \Id$. Explicitly, the action of $\mathcal{L}$ is given by $\mathcal{L}(A\cdot \alpha) = \mathcal{L}(A)\cdot \alpha$. Now suppose that   $\alpha \in \bigwedge^mW_1$ satisfies $\star \alpha = t\alpha$ for some $t\in E$. It is easy to see that this happens  $\iff$ $(-1)^mD = \Nm_{E/E_0}(c)$ for some $c\in E$. Then it is straightforward to check that     $\mathcal{L}(I_\alpha^{\text{\rm{even}}}) = I_\alpha^{\text{\rm{even}}}$ and hence that there is an induced action of $\mathcal{L}$ on $S^\pm$. Explicitly, for $\eta \in \bigwedge^{m+2d}W_1$, and using $(-1)^mD = t\bar{t}$,
$$\mathcal{L}(\eta) = (-1)^{m+d}\lambda^{-d}D^{-1}\bar{t} (\star \eta) = (-1)^d\lambda^{-d}(t^{-1}  \star )(\eta).$$
Thus, up to the scalar $t^{-1}\in E^*$, the operator $\mathcal{L}$ is exactly $L_\pm$ depending on the parity of $m$. Using  $S^\mp = C^{\text{\rm{odd}}}(W,b)\cdot \alpha$, a  similar argument shows that, if $\star \alpha = t\alpha$ for some $t\in E$, then  the operator $\mathcal{M}$ preserves the $C^{\text{\rm{even}}}(W,b)$-submodule $I^{\text{\rm{odd}}}_\alpha$ of $C^{\text{\rm{odd}}}(W,b)$ given by
$$I^{\text{\rm{odd}}}_\alpha =\{A\in C^{\text{\rm{odd}}}(W,b): A\cdot \alpha = 0\}.$$ Then,  as before, $\mathcal{M}$ defines a conjugate linear operator on $S^\mp$ which   commutes with the $G_1$-action and which is $L_\mp$ up to multiplication by an element of $E^*$. Thus we reprove the fact that $L_+$ and $L_-$ commute  with the action of $G_1$ in a slightly more conceptual fashion,  under the assumption that $(-1)^mD$ is a norm. Finally, if $(-1)^mD$ is not a norm, we can pass to quadratic extensions $E_0'$ and $E'=EE_0'$ of $E_0$ and $E$ respectively, such that $(-1)^mD$ is a square in $E_0'$ and hence lies in $\Nm_{E'/E_0'}(E')^*$. There is then an  operator induced by $\mathcal{L}$ or $\mathcal{M}$ on $(\Res_{E/E_0}S^\pm) \otimes E_0' = \Res_{E'/E_0'}(S^\pm\otimes _EE')$, and it  is a multiple of $L_\pm\otimes \Id_{E_0'}$ by a nonzero element of $(E')^*$. Thus $L_\pm \otimes \Id_{E_0'}$ commutes with the action of $G_1$, and hence the same must be true for $L_\pm$. This then gives another proof of Theorem~\ref{allcases} in general.
\end{remark}

\bibliography{refcy}

\end{document}